\documentclass[journal,twoside,web]{article}

    \usepackage{amsmath,amssymb,amsfonts
    }  
    \usepackage{graphicx,wrapfig}

\usepackage{enumitem}
\usepackage{xcolor}
    \def\BibTeX{{\rm B\kern-.05em{\sc i\kern-.025em b}\kern-.08em
        T\kern-.1667em\lower.7ex\hbox{E}\kern-.125emX}}
  \usepackage[pdftex, pdfborderstyle={/S/U/W 0}]{hyperref}

\newtheorem{exe}{Example}
\newtheorem{corol}{Corollary}
\newtheorem{ass}{Assumption}
\newtheorem{defin}{Definition}
\newtheorem{prob}{Problem}
\newtheorem{cla}{Claim}
\newtheorem{rem}{Remark}
\newtheorem{lem}{Lemma}
\newtheorem{prop}{Proposition}
\newtheorem{thm}{Theorem}
\newtheorem{fct}{Fact}
\newenvironment{lemma}{\begin{lem}}{\hfill $\square$ \end{lem}}

\newenvironment{corollary}{\begin{corol}}{\hfill $\square$ \end{corol}}

\newenvironment{remark}{\begin{rem} \rm}{\hfill $\bullet$ \end{rem}}
\newenvironment{assumption}{\begin{ass}}{\hfill $\bullet$ \end{ass}}
\newenvironment{theorem}{\begin{thm}}{\hfill $\square$ \end{thm}}

\newenvironment{proof}[1][\!]{{\it Proof\,#1: }}{{\hfill \mbox{\footnotesize$\blacksquare$}}}

\begin{document}
\title{\LARGE \bf Global Uniform Ultimate Boundedness of Semi-Passive Systems Interconnected over Directed Graphs}


\author{Anes Lazri \quad Mohamed Maghenem \quad Elena Panteley \quad Antonio Lor\'{\i}a 
  \thanks{M. Maghenem is with University of Grenoble Alpes,  CNRS,  Grenoble INP,  GIPSA-Lab,  France. 
        E-mail: mohamed.maghenem@cnrs.fr; 
        E. Panteley and A. Lor\'ia are with L2S, CNRS, 91192 Gif-sur-Yvette, France. 
        E-mail: elena.panteley@cnrs.fr and antonio.loria@cnrs.fr 
        A. Lazri  is with L2S, CNRS, Univ Paris-Saclay, France 
          (e-mail: anes.lazri@centralesupelec.fr)
        }}

\maketitle

\begin{abstract}
We analyse the solutions of networked heterogeneous nonlinear systems\footnote{For simplicity, but without loss of generality, we assume that $x\in \mathbb{R}$; all statements hold after pertinent changes in the notation, if $x\in \mathbb{R}^p$, with $p >1$.}
\begin{equation}\label{120}
\dot x_i = f_i(x_i) + u_i \qquad  x_i \in \mathbb{R}, \quad  i \in \{1,2,\cdots,n\},
\end{equation}
where $f_i : \mathbb{R} \rightarrow \mathbb{R}$ is continuous for all $i \in \{1,2,\cdots, n\}$ and the control inputs are set to  
\begin{align} \label{134} 
u_i := - \gamma \sum_{i=1}^n a_{ij} (x_i - x_j)    \qquad \forall  i \in \{1,2,\cdots,n\},  
\end{align} 
where $\gamma>0$ is a coupling gain and $a_{ij}\geq 0$  are interconnection weights. We assume that the closed-loop interconnected systems form a network with an underlying connected directed graph that contains a directed spanning tree. For these systems, we establish global uniform ultimate boundedness of the solutions, under the assumption that each system \eqref{120} defines a semi-passive \cite{Pogr3} map $u_i \mapsto x_i$. As a corollary, we also establish global uniform global boundedness of the solutions.
\end{abstract}

\section{Preliminaries}\label{sec1}

\textbf{Notations.} For $x  \in \mathbb{R}^n$,  $x^\top$ denotes its transpose,  $|x|$ denotes its Euclidean norm,  
$\text{blkdiag}\{x\} \in \mathbb{R}^{n \times n}$ denotes the diagonal matrix whose $i$th diagonal element is the $i$th element of $x$.   For a set $K \subset \mathbb{R}^n$,  $|x|_K := \min \{|x-y| : y \in K \}$ denotes the distance of $x$ to the set $K$. For a symmetric matrix $Q \in \mathbb{R}^{n \times n}$,  $\lambda_i(Q)$ denotes the $i$th smallest eigenvalue of $Q$.  For an invertible  matrix $M \in \mathbb{R}^{n \times n}$, $M^-$ or $M^{-1}$ denotes its inverse.  Given  $N \in \mathbb{R}^{n \times n}$, $\text{Ker}(N) := 
\{v : N v =0\}$ denotes the kernel of $N$.  A class $\mathcal{K}^\infty$ function 
$\alpha : \mathbb{R}_{\geq 0} \rightarrow \mathbb{R}_{\geq 0}$ is continuous, strictly increasing, unbounded, and $\alpha(0) = 0$. Furthermore
$\alpha^-$ denotes the inverse function of $\alpha$.

\subsection{On Some Classes of Matrices}

A matrix $M := [m_{ij}]$,  $(i,j) \in  \{1,2,...,n\}^2$, is a $\mathcal{Z}$-matrix if $m_{ij} \leq 0$ whenever $i \neq j$.  
It is an $M$-matrix if it is a 
$\mathcal{Z}$-matrix and its eigenvalues have 
non-negative real parts.  Equivalently,  $M := \lambda I_n - B$,  where $B$ is a non-negative matrix and $\lambda \geq  \rho(B)$,  where $\rho(B) := \max \left\{ |\lambda_i(B)| : i\in \{1,2,...,n\}  \right\}$ is the spectral radius of $B$.   
$M$ is a non-singular $M$-matrix if it is a $\mathcal{Z}$-matrix and its eigenvalues have positive real parts.  
Equivalently,  $M := \lambda I_n - B$,  where $B$ is a non-negative matrix and $\lambda >  \rho(B) > 0$; 
see \cite{7337394,9312975} for more details.

\subsection{Graph Notions}
A directed graph or a digraph $\mathcal{G}(\mathcal{V},\mathcal{E})$ is characterized by the set of nodes $\mathcal{V} = \{1,2,...,n\}$, and the set of directed edges $\mathcal{E}$.  The edge set $\mathcal{E}$ consists of ordered pairs, of the form $(k, i)$, that indicate a directed link from node $k$ to node $i$.  Given a directed edge $(k, i) \in \mathcal{E}$, then node $k$ is called an \emph{in-neighbor} of node $i$. We assign a positive weight $a_{ik}$ to each edge $(k, i)$. That is, $a_{ik} = 0$ if $(k,i)$ is not an edge. The Laplacian matrix of a digraph is given by 
\begin{equation}
  \label{126} L  :=  
\begin{bmatrix}
 d_1  & - a_{12}   & \cdots & - a_{1n}  
\\
- a_{21}  &   d_{2}   & \cdots & - a_{2n}  
\\
\vdots &  \vdots  & \vdots & \vdots
\\
- a_{n-1 1}  &  \cdots & d_{n-1}  & - a_{n-1 n} 
\\
- a_{n 1}  &  \cdots & - a_{n n-1}   &   d_{n} 
\end{bmatrix} =: D - A,  
\end{equation}
where $d_i := \sum^n_{j=1}  a_{ij}$ for all $i \in \{1,2,...,n\}$,  $D$ is the diagonal part of $L$ and $A$ is called the adjacency matrix.   

A digraph is \textit{strongly connected} if, for any two distinct nodes $i$ and $j$, there is a path from $i$ to $j$. The Laplacian matrix of a strongly connected graph admits $\lambda_1(L)=0$ as an eigenvalue with the corresponding right and left eigenvectors
$1_n = \begin{bmatrix} 1 ~ 1 ~ \cdots ~ 1 \end{bmatrix}^\top$ and   
$v_{o} := \begin{bmatrix} v_1 ~ v_2 ~ \cdots ~ v_n \end{bmatrix}^\top$, respectively, 
where  $v_i >0$ for all $i \leq n$.

\subsection{Graph and Matrix Decomposition}
Suppose that the digraph $\mathcal{G}$ is connected and contains a spanning tree. Then, it admits a decomposition into a leading strongly connected subgraph  $\mathcal{G}_\ell \neq \O$ and a  subgraph  $\mathcal{G}_f:= \mathcal{G} \backslash \mathcal{G}_l$ of followers; namely,  the agents that do not belong to the leading component, and which we call the follower agents.  In this case,  up to a permutation, the Laplacian $L$ admits the lower-block decomposition
\begin{align} \label{eqWCdec} 
L = \begin{bmatrix} 
 L_\ell  &  0 \\  - A_{\ell f} & M_f \end{bmatrix},  
 \end{align}
where $L_\ell := D_\ell - A_\ell  \in \mathbb{R}^{n_\ell \times n_\ell}$ is the Laplacian matrix of the strongly connected component 
$\mathcal{G}_\ell$,  the lower-left block $A_{\ell f} \in \mathbb{R}^{n_f \times n - n_f}$, $n_f := n - n_\ell$,  is a non-negative matrix,  and the lower-right block $M_f \in \mathbb{R}^{n_f \times n_f}$ is a non-singular M-matrix.  The block $M_f$ can be seen as the sum of the Laplacian matrix $L_f$ corresponding to $\mathcal G_f$ and a diagonal matrix $D_{\ell f}$ gathering the weights of the interconnections between nodes in $\mathcal G_\ell$ and the nodes in $\mathcal G_f$.  That is,   $M_f = L_f + D_{\ell f}$,  where $L_f = D_f - A_f$.   

\subsection{Lyapunov Analysis of a Directed Graph}
Consider a network of $n$  single integrators of the form $\dot{x}_i = u_i$ interconnected according to the classical consensus protocol 
$$ u_i := -  \sum_{i=1}^n a_{ij} (x_i - x_j)    \qquad \forall  i \in \{1,2,\cdots,n\}. $$

In closed loop, the network is governed by  the linear  system 
$ \dot{x} = - L x$,   where $L \in \mathbb{R}^n$ is the Laplacian matrix of a connected  di-graph $\mathcal{G}$ that contains a  directed spanning tree. According to Section \ref{sec1},   we can decompose the state $x$ into  $x^\top := \left[ x_l^\top \ x_f^\top \right]$, where
$x_l \in \mathbb{R}^{n_l}$ gathers the states of the leading component and is governed by 
\begin{align*} 
 \Sigma_\ell: \, \dot{x}_\ell = - L_\ell x_\ell,  
\end{align*} 
and the non-leading component whose state is $x_f \in \mathbb{R}^{n_f}$,  are governed by 
\begin{align*} 
\Sigma_f: \, \dot{x}_f =  - M_f   x_f,
\end{align*} 
on the manifold $\{ x_f = 0 \}$.
In the rest of this section, we overview some Lyapunov-function constructions allowing to prove uniform exponential stability of $\mathcal{A}$ for $ \Sigma_\ell$,  where  
\begin{align} \label{eqAell}
 \mathcal{A}:= \{ x_l \in \mathbb{R}^{n_l} : x_{l1} = x_{l2} = \cdots = x_{l_{n_\ell}} \},
 \end{align}
and   exponential stability of the origin for $ \Sigma_f$.

\subsubsection{Proof of uniform exponential stability of $\mathcal{A}$ for $ \Sigma_\ell$}
let $v_{o} := \begin{bmatrix} v_1 ~ v_2 ~ \cdots ~ v_{n_\ell} \end{bmatrix}^\top$ be a left eigenvector associated to $\lambda_1(L_\ell)=0$ and $V_o :=\text{blkdiag} \{v_o \}$.
Based on Lemma \ref{lem1} in the Appendix,  $Q_o:= L_\ell^\top V_{o}  + V_{o} L_\ell$  is symmetric and positive semi-definite,  and its kernel is spanned by $\boldsymbol 1_{n_\ell}$.  Then, the derivative of the Lyapunov function candidate 
$W(x_\ell) := x^\top_\ell V_{o} x_\ell$,  along the solutions to $ \Sigma_\ell$,  satisfies
$$ \dot W(x_\ell)= - x^\top_\ell ( L_\ell^{\top} V_{o}  + V_{o} L_\ell) x_\ell  \leq -\lambda_2 (Q_o) |x_\ell|_{\mathcal{A}_\ell}^2. $$
 Now, we let 
$$ Z(x_\ell) :=  \left( x_\ell - \boldsymbol 1_{n_\ell} v_{o}^\top  x_\ell \right)^\top V_{o} \left( x_\ell -  \boldsymbol 1_{n_\ell} v_{o}^\top  x_\ell \right),  $$ 
which is positive definite. Its derivative along the solutions of $\dot x_\ell = - L_\ell x_\ell$ satisfies 
 \begin{align} \label{eqDeltaZs}
\dot Z(x_\ell)= - x^\top_\ell Q_{o} x_\ell  \leq 
- \lambda_2(Q_{o})  |x_\ell|^2_{\mathcal{A}_\ell}.  
\end{align}
To obtain the previous expression we used $v_o^\top L = 0$, $v_{1}^\top \boldsymbol 1_{n_\ell} =1$ and that $\boldsymbol 1_{n_\ell}$ is in the kernel of $I_{n_\ell} - \boldsymbol 1_{n_\ell} v_{o}^\top$.  Moreover,   $I_{n_\ell} - \boldsymbol 1_{n_\ell} v_{o}^\top$ is the Laplacian matrix of an all-to-all graph; hence,  $\boldsymbol 1_{n_s}$ spans the kernel of $I_{n_\ell} - \boldsymbol 1_{n_\ell} v_{o}^\top$.  Therefore, there exist \color{black} $\bar{z}$,  $\underline{z} >0$ such that 
\begin{align}  \label{eqZpropers}
 \underline{z} |x_\ell|^2_{\mathcal{A}_\ell}   \leq  Z(x_\ell) \leq \bar{z} |x_\ell|^2_{\mathcal{A}_\ell} \qquad \forall x_\ell \in \mathbb{R}^{n_\ell}.   
 \end{align}
Uniform exponential stability of $\mathcal{A}_\ell$ from \eqref{eqDeltaZs} and \eqref{eqZpropers} and standard Lyapunov-stability theory.  \color{black}

\subsubsection{ Proof of exponential Stability of the Origin for $ \Sigma_f$}
based on Lemma \ref{lem2Mf},  since $M_f$ is a non-singular $M$-matrix, we can use the Lyapunov function candidate $Y(x_f) :=  x^\top_f  R_f x_f$,  where 
$R_f := \text{blkdiag} \left\{ {M_f}^{-\top}
1_{n_f} \right\}   \left( \text{blkdiag} \left\{ M_f^{-1} 
1_{n_f} \right\} \right)^{-1} $,  which is positive definite.  
Furthermore,  along the solutions  to $ \Sigma_f$, we have 
$$\dot Y(x_f)=  - x^\top_f [ M^{\top}_f R_{f}  + R_f M_f] x_f.$$
Now, since $( M^{\top}_f R_{f}  + R_f M_f)$ is positive definite,  exponential stability of the origin for $\Sigma_f$ follows.

\section{Problem formulation}

Consider the systems \eqref{120}-\eqref{134}, with $\gamma>0$ and $a_{ij} \geq 0$. 
Then, defining $x := [x_1\ \cdots \ x_n]^\top$, and $F(x) := \big [f_1( x_1),f_2( x_2), \cdots, f_n(x_n) \big ]^\top$, we may write the closed-loop system in compact form as
\begin{equation} \label{eqcl}
\dot x = F(x) - \gamma L x,
\end{equation}
where $L$ is defined as in \eqref{126}. This is a networked system with an underlying topology that may be represented by a graph  $\mathcal{G}$. 
\begin{assumption} \label{ass0}
The graph $\mathcal G$ is connected and contains a directed spanning tree. 
\end{assumption}
\color{black}

We are interested in verifying the following two boundedness properties for \eqref{eqcl}. 

\begin{enumerate}[label={(P\arabic*)},leftmargin=*]
\item  \label{itemP1}  \textit{Global Uniform Boundedness} (GUB).    
The solutions $t \rightarrow x(t)$ to \eqref{eqcl} are globally bounded, uniformly in $\gamma$,  if, for every $r_o > 0$ and $\gamma_o > 0$,  there exists $R = R(r_o, \gamma_o) \geq r_o$ such that,  for all 
$\gamma \geq \gamma_o$,
\[
|x(t_o)| \leq r_o \Rightarrow |x(t)| \leq R  \quad \forall t  \geq 0.
\]
\item \label{itemP2}  \textit{Global Uniform Ultimate Boundedness} (GUUB).  The solutions $t \rightarrow x(t)$ to \eqref{eqcl} are ultimately bounded, uniformly in $\gamma$,  if given $\gamma_o > 0$,  there exists $r =  r(\gamma_o)>0$ such that,  for all $r_o >0$,  there exists $T = T(r_o,\gamma_o) \geq 0$ such that, for all $\gamma \geq \gamma_o$,
\[ |x(t_o)| \leq r_o \Rightarrow |x(t)| \leq r \quad \forall t \geq  T.
\]
\end{enumerate}

To verify the latter two properties,  we make the following assumption on the individual nodes' dynamics in \eqref{120}. 

\begin{assumption}  [State strict semi-passivity] \label{ass1}
For each  $i \in \{1,2,...,n\}$,  the input-output map $u_i \mapsto x_i$ defined by the dynamics \eqref{120} is state strict semipassive \cite{IJBC_Pogromsky}. Furthermore,  there exists a continuously differentiable storage function $V_i : \mathbb{R}^n \rightarrow \mathbb{R}^+$, 
a class $\mathcal{K}_\infty$ function $\underline{\alpha}_i$,  a constant
$\rho_i > 0$,  a continuous function $H_i : \mathbb{R} \rightarrow \mathbb{R}$,  and a  continuous function $\psi_i : \mathbb{R}_{\geq 0} \rightarrow \mathbb{R}_{> 0}$, such that
\begin{equation}\label{128}
\underline{\alpha}_i(|x_i|) \leq V_i(x_i), \quad    \dot{V}_i(x_i) \leq 2 u_i x_i - H_i(x_i),  
\end{equation}
and $H_i(x_i) \geq \psi_i(|x_i|)$ for all $|x_i| \geq \rho_i$.
\end{assumption}

\begin{remark}
The property described in Assumption \ref{ass1} is called strict {\it quasipassivity} in \cite{POLUSHIN1998505}. In \cite{Pogr3} the authors define a similar concept named strict semi-passivity, but radial unboundedness of the storage function is not imposed. See also \cite{IJBC_Pogromsky}.
%
%
\end{remark}

\section{Main result}

\begin{theorem}[Uniform ultimate boundedness]  \label{thm1}
The solutions of the networked system \eqref{120}-\eqref{134} are globally uniformly ultimately bounded, {\it i.e.,} Property  \ref{itemP2} holds, if Assumptions \ref{ass0} and \ref{ass1} are satisfied.
\end{theorem}\color{black}
\begin{proof}
Under Assumption \ref{ass0}, the Laplacian matrix $L$ admits a permutation, such that  \eqref{eqWCdec} holds. Therefore, the state $x$ may be decomposed into $x := [x_\ell^\top\ x_f^\top]^\top$ and the system \eqref{eqcl} takes the cascaded form 
\begin{subequations}
  \label{273}
\begin{align} \label{eqLead}
\dot x_\ell & = f_\ell (x_\ell ) - \gamma L_\ell x_\ell,  & f_\ell(x_\ell) := \begin{bmatrix}f_1(x_{\ell_1})& \cdots& f_{n_\ell}(x_{\ell _{n_\ell}}) \end{bmatrix}^\top\\
\label{foldyn}
\dot x_f &= f_f(x_f) + \gamma A_{\ell f} x_\ell - \gamma M_f x_f, &\qquad  f_f(x_f) := \begin{bmatrix} f_{n_\ell+1}(x_{f_1}) & \cdots& f_{n_\ell+n_f}(x_{f _{n_f}})  \end{bmatrix}^\top.
\end{align}
\end{subequations}
Equation \eqref{eqLead} corresponds to the dynamics of a {\it leading} component, a networked system with an underlying strongly connected graph $\mathcal G_\ell$, and a {\it follower} component, with dynamics \eqref{foldyn}. The proof of the statement is constructed using a cascades argument and proving, firstly, global uniform  ultimate boundedness for the solutions of \eqref{eqLead} and, consequently, the same property for \eqref{foldyn}.

To that end, let $r_o > 0$ be arbitrarily fixed and let $|x(0)| \leq r_o$. Then, $|x_\ell(0)| \leq r_o$ and $|x_f(0)| \leq r_o$.   
 \color{black}

\noindent 1) {\it Uniform ultimate boundedness for  the leading component}: 
after  Assumption \ref{ass1}, for each $i \in \{1,2,..., n_\ell\}$, 
there exists a storage function $V_{i}$  such that its total derivative along the trajectories of \eqref{120} satisfies \color{black} 
\begin{equation}\label{217}
\dot V_{i}(x_{\ell i})  \leq 2  u_i^\top x_{\ell i} - H_i (x_{\ell i}), \quad H_i(x_{\ell i}) \geq \psi_i(|x_{\ell i}|) \quad \forall |x_{\ell i}| \geq \rho_i.
\end{equation}

Next, let \color{black}$W(x_\ell) :=  \sum_{i=1}^{n_\ell} v_i V_i(x_{\ell i})$, where $v_i$ corresponds to the $i$th element of $v_o$, which is \color{black} the left eigenvector associated to the zero eigenvalue of $L_\ell$.  Since the graph $\mathcal{G}_\ell$ is strongly connected,  then $v_i > 0$ for all $i \leq n_\ell$, so $W$ is positive definite and radially unbounded. \color{black}  Now,  from \eqref{217}, we obtain \color{black}
\begin{equation}\label{222}
\dot W (x_{\ell}) \leq  2 \sum_{i=1}^{n_\ell} v_i u_i^\top x_{\ell i} -  \sum_{i=1}^N v_i H_i (x_{\ell i}), \qquad \forall x_\ell \in \mathbb R^{n_\ell}. \color{black}
\end{equation}
The first term on the right-hand side of \eqref{222} satisfies
\begin{equation}\label{226}
 \sum_{i=1}^{n_\ell} v_i u_i^\top x_{\ell i}  = u^\top V_o  x_\ell,
\end{equation}
where $V_o:=\mbox{blkdiag}\{v_o\}$ and, since $u=-\gamma L_\ell x_{\ell}$,  it follows that 
\begin{align}\label{248}
\nonumber 
\dot W (x_{\ell}) &\leq  -  \sum_{i=1}^{n_\ell} v_i H_i (x_{\ell i}) - \gamma x_\ell^\top [ L_\ell^\top V_o +V_o L_\ell]  x_\ell\\
 &\leq  -   \sum_{i=1}^{n_\ell} v_i H_i (x_{\ell i}) - \gamma x_\ell^\top Q_o x_\ell,
\end{align}
with $Q_o := V_o L_\ell + L_\ell^\top V_o$,  which is positive semi-definite---see Lemma \ref{lem1} in the Appendix.  Furthermore,  we note that 
$$- x_\ell^\top Q_o x_\ell =  - \left[x_\ell - \boldsymbol 1_{n_\ell} \boldsymbol 1_{n_\ell}^\top x_\ell/n_\ell \right]^\top Q_o \left[x_\ell - \boldsymbol 1_{n_\ell} \boldsymbol 1_{n_\ell}^\top x_\ell/n_\ell \right] \leq - \lambda_2(Q_o) |x_\ell|_{\mathcal A}^2, $$ 
where $|x_\ell|_{\mathcal A}$ denotes the distance of $x_\ell$ to the set $\mathcal{A}$ and  
$\lambda_2(Q_o)$ is the second smallest eigenvalue of $Q_o$.

Now, on one hand, we have that $v_i >0 $ for all $i\in \{1,2,\ldots,n\}$ and, on the other,  $-H_i(x_{\ell i}) > 0$  only if $|x_{\ell i}| \leq \rho_i$. Therefore,  the constant \color{black}
$  H_\ell := - \sum_{i=1}^{n_\ell}  \max_{|x_i| \leq \rho_i}  \left\{ v_i H_i\big (x_{\ell i} \big ) \right\} > 0$.   
Therefore, after \eqref{248}, we get \color{black}
\begin{align} \label{eqWdot}
 \dot W (x_{\ell}) &\leq H_\ell -  \gamma \lambda_2(Q_o) |x_\ell|_{\mathcal A}^2 \qquad \forall \, x_\ell \in \mathbb R^{n_\ell}. \color{black}
\end{align}
In turn,   given $\gamma_o>0$ and $\epsilon >0$,  for all $\gamma \geq \gamma_o$,  we have
\begin{align}\label{257}
\dot W  (x_{\ell}) & \leq H_\ell -  \gamma_o \lambda_2(Q_o) |x_\ell|_{\mathcal A}^2 \leq -\epsilon   \qquad \forall x_\ell \notin \mathcal{C}, 
\end{align}
where 
$$\mathcal C := \left\{ x_\ell \in \mathbb{R}^{n_\ell} : |x_\ell|_{\mathcal A} \leq \sqrt{n_\ell} R_e  :=  \sqrt{\frac{\epsilon + H_\ell}{\gamma_o  \lambda_2(Q_o)}} \right\}.  $$ 

Next, let  $\bar \rho := \text{arg}\hspace{-.9em}\displaystyle\max_{\hspace{-1em}i \in \{1,2,...,n_\ell\}} \rho_i $ and \color{black}
$$ \mathcal B_\beta := \{ x_\ell \in \mathbb{R}^{n_\ell} : |x_\ell| \leq \beta :=  \sqrt{n_\ell} \big (\bar \rho + 2 R_e) \}.  $$ 
Note that for all $x_\ell \notin \mathcal B_\beta$, we have 
 $|x_\ell| >  \sqrt{n_\ell}  (\bar \rho + 2R_e)$
and, for all $x_\ell \in \mathcal{C} \backslash \mathcal{B}_\beta$,  
\begin{align} \label{eqbounds} 
|x_\ell| >  \sqrt{n_\ell}  (\bar \rho + 2 R_e) \quad \text{and} \quad |x_\ell|_\mathcal{A} \leq \sqrt{n_\ell} R_e.   
\end{align}
Furthermore,  we use the fact that 
$ x_\ell  =  \boldsymbol 1_{n_\ell} (\boldsymbol 1_{n_\ell}^\top  x_\ell)/n_\ell  + \left[ x_\ell - \boldsymbol 1_{n_\ell} (\boldsymbol 1_{n_\ell}^\top  x_\ell)/n_\ell \right]$,  and the fact that $|x_\ell|_\mathcal{A} = |x_\ell - \boldsymbol 1_{n_\ell} (\boldsymbol 1_{n_\ell}^\top  x_\ell)/n_\ell|$,  to conclude that
\begin{align} \label{eqbounds1}
   | x_\ell | \leq |x_\ell|_\mathcal{A}  +  | \boldsymbol 1_{n_\ell}^\top  x_\ell| / \sqrt{n_\ell}. 
\end{align}
Now,  combining \eqref{eqbounds} and \eqref{eqbounds1}, we conclude that for all \color{black}$x_\ell \in \mathcal{C} \backslash \mathcal{B}_\beta$, 
\begin{align} \label{eqbounds2}
\sqrt{n_\ell}  (\bar{\rho} + 2 R_e) <   | x_\ell | \leq |x_\ell|_\mathcal{A}  + | \boldsymbol 1_{n_\ell}^\top  x_\ell| / \sqrt{n_\ell}  \leq \sqrt{n_\ell} R_e + | \boldsymbol 1_{n_\ell}^\top  x_\ell|/  \sqrt{n_\ell}. 
\end{align}
So, for all $x_\ell \in \mathcal{C} \backslash \mathcal{B}_\beta$, 
$  | \boldsymbol 1_{n_\ell}^\top  x_\ell|/n_\ell >  \bar{\rho} + R_e$. 
Next,  we use the fact that  
$$ x_{\ell i} = \boldsymbol 1_{n_\ell}^\top  x_\ell/n_\ell  + \left( x_{\ell i}  - \boldsymbol 1_{n_\ell}^\top  x_\ell / n_\ell \right) \quad \forall i \in\{1,2,...,n_\ell\} $$
to conclude that 
$|x_{\ell i}| > \color{black}|\boldsymbol 1_{n_\ell}^\top  x_\ell| / n_\ell - |\left( x_{\ell i}  - \boldsymbol 1_{n_\ell}^\top  x_\ell/n_\ell \right)|$. Hence, 
$$|x_{\ell i}| > \bar{\rho} +  R_e  - \sqrt{n_\ell} R_e > \bar{\rho} \qquad \forall i \in \{1,2,...,n_\ell \}   $$
for all \color{black}$x_\ell \in \mathcal{C} \backslash \mathcal{B}_\beta$. The latter,  under Assumption \ref{ass1},   implies that
\begin{align}
\nonumber 
-   \sum_{i=1}^{n_\ell} v_i H_i (x_{\ell i})   \leq -   \sum_{i=1}^{n_\ell} v_i \psi_i (|x_{\ell i}|)  \leq  0 \qquad \forall x_\ell \in \mathcal{C} \backslash \mathcal{B}_\beta.
\end{align}
As a result,  setting $\Psi (x_\ell) := \sum_{i=1}^{n_\ell} v_i \psi_i (|x_{\ell i}|)$---note that $\Psi$ is continuous and positive---we conclude that
$$
 \dot W (x_{\ell}) \leq  - \Psi (x_\ell)  \qquad \forall x_\ell \in \mathcal{C} \backslash \mathcal{B}_\beta.
$$
Combining the latter inequality to \eqref{257},  we conclude that
$$
 \dot W (x_{\ell}) \leq  - \min \{ \Psi (x_\ell),  \epsilon \}  \qquad \forall x_\ell \in \mathbb{R}^{n_\ell}  \backslash \mathcal{B}_\beta.
$$
The latter is enough to conclude global attractivity and forward invariance of the set 
\begin{align*} 
\mathcal{S}_\sigma := \{ x_\ell \in \mathbb{R}^{n_\ell} : W(x_\ell) \leq \sigma \}, \quad \sigma :=  \max \{ W(y) : y \in \mathcal{B}_\beta \}.  
\end{align*}
Furthermore, since $W : \mathbb{R}^n \rightarrow \mathbb{R}_{\geq 0}$ is continuous and $\mathcal{B}_\beta$ is bounded, we conclude that $\sigma$ is well defined. Consequently, the ultimate bound  is 
$$ r_\ell := \left[  \min_i \{\underline{\alpha}_i\} \right]^- (\sigma), $$
where, with an abuse of notation, $\min_i \{\underline{\alpha}_i\}$ corresponds to the function $s\mapsto \psi(s)$ defined as  $ \psi(s) := \min_i \{\underline{\alpha}_i(s)\}$ for each $s \geq 0$ and $\underline{\alpha}_i$ is defined in Assumption \ref{ass1}, so $\psi: \mathbb R_{\geq 0}\to \mathbb R_{\geq 0}$ is strictly increasing and radially unbounded, hence, globally invertible.  Thus,  $W(x_\ell) \leq \sigma$ implies that $|x_\ell| \leq r_\ell$. \color{black}

Next, we compute an upperbound \color{black} $T_\ell (r_o,\gamma_o)$ on \color{black} the time that the solutions to \eqref{eqLead}, with $\gamma \geq \gamma_o$ and starting from $\mathcal{B}_{r_o} := \{x_\ell \in \mathbb{R}^{n_\ell} : |x_\ell| \leq r_o \}$, take to reach the compact set \color{black} $\mathcal{B}_\beta \subset \mathcal{S}_\sigma$. For this, we assume without loss of generality that $r_o \geq \beta$, and we define
$$ \epsilon_{r_o} := \min \{ \min \{ \Psi (x_\ell),  \epsilon \}  : |x_\ell | \geq \beta, ~ x_\ell \in \mathcal{S}_{\sigma_o} \},  $$
where 
\begin{align} \label{eqSsigmao} \mathcal{S}_{\sigma_o} := \{ x_\ell \in \mathbb{R}^{n_\ell} 
: W(x_\ell) \leq \sigma_o
\}, ~ \sigma_o := \max \{ W(y) : y \in \mathcal{B}_{r_o} \}. 
\end{align}
Clearly, $\mathcal{S}_{\sigma_o}$ is compact; hence, $\epsilon_{r_o}$ is positive. 

Therefore, along every solution $t \mapsto x_\ell(t)$ to \eqref{eqLead} starting from 
$x_{\ell} (0) \in \mathcal{B}_{r_o} \backslash \mathcal{B}_{\beta}$, we have 
$\dot W (x_{\ell}(t)) \leq  - \epsilon_{r_o}$, up to the earliest time when $x_\ell$ reaches $\mathcal{B}_{\beta}$. For any earlier time, \color{black} 
we have 
\begin{equation}\label{338}
W(x_{\ell}(t)) \leq  - \epsilon_{r_o} t + W(x_{\ell}(0)),
\end{equation} 
so we can take $T_\ell(r_o,\gamma_o) = \sigma_o/\epsilon_{r_o}$. 
Clearly, $T_\ell$ depends only on $(r_o,\gamma_o)$ and $r_\ell$ depends only on $\gamma_o$. Thus, the ultimate bounded  guaranteed for the solutions of \color{black} \eqref{eqLead} is uniform in $\gamma$.

\noindent 2) {\it  Uniform ultimate boundedness for the follower dynamics}:
following up the previous computations and arguments, 
establish global uniform ultimate boundedness for the non-leading component, determined by  \eqref{foldyn}.\color{black}

Using Lemma \ref{lem2Mf}, we conclude that the matrices 
$$ S := P M_f + M_f^\top P \quad  \text{and} \quad  P :=  \text{blkdiag} \left\{ {M_f^\top}^- 
1_{n} \right\}   \left( \text{blkdiag} \left\{ M_f^- 
1_{n} \right\} \right)^- $$ 
are symmetric and positive definite. We also note that $P$ is diagonal. Then, let $p_i$, for $i \in \{1,2,..., n_f\}$, be the $i$th diagonal element of $P$. In addition,  let $Z (x_f) := \sum_{i=1}^{n_f} p_i V_i (x_{f i})$. Its total derivative along the trajectories of \eqref{foldyn} satisfies
\begin{align}\label{405}
\dot Z(x_f) \leq - \sum_{i=1}^{n_f} p_i H_i(x_{ f i})- \gamma x_f^\top [  P M_f + M_f^\top P ] x_f + 2 \gamma x_f^\top [ P A_{\ell f} ] x_\ell.
\end{align}

On one hand, \color{black}we already established the existence of $r_\ell(\gamma_o) > 0$ and $T_\ell(\gamma_o,r_o)$ such that 
$$  |x_\ell(t)|<r_\ell \qquad  \forall t \geq T_\ell. $$
On the other, for all \color{black} $|x_\ell| \leq r_l$, 
\begin{align} \label{296}
\dot Z (x_f) \leq  H_f - \gamma \lambda_{\text{1}}(S) |x_f|^2 + 2 \gamma \bar p r_\ell |x_f|, 
\end{align}
where $\bar p :=| P A_{\ell f}|$ and $H_f := \sum_{i=1}^{n_f}  \max_{|x_i| \leq \rho_i}  \{v_i H_i\big (x_{f i} \big ) \}$. Now, from this and \eqref{405}, we obtain \color{black}
\begin{align*}
\dot Z(x_f) & \leq H_f - \gamma x_f^\top S x_f + 2 \gamma x_f^\top [ P A_{\ell f} ] x_\ell
\\ &
\leq H_f - \gamma \left[  x_f^\top S x_f/2 - 2  x_\ell^\top  A_{\ell f}^\top P^\top S^- P A_{\ell f}  x_\ell \right].
\end{align*}
At the same time, integrating \eqref{eqWdot}, we obtain that,  for each $t \in [0,T_\ell]$,  
$$  W(x_\ell(t)) \leq H_\ell T_\ell + W(x_\ell(0)) \leq  H_\ell T_\ell + \sigma_o, $$
where $\sigma_o$ comes from \eqref{eqSsigmao}.  
Defining 
$$ R_\ell := \left[ \min_i 
\{\underline{\alpha}_i\}  \right]^- \left(H_\ell T_\ell + \sigma_o \right),$$
we have  
\begin{align*}
\dot Z(x_f) & 
\leq H_f - \gamma 
\left[ \lambda_{1}(S) |x_f|^2/2 - 2   |A_{\ell f}^\top P^\top S^- P A_{\ell f}|  R_\ell^2 \right],
\end{align*}
for all 
$|x_\ell| \leq R_\ell$.

Note that, for all $x_f$ such that 
$$ |x_f|^2 > d^2_f :=  \frac{4 |A_{\ell f}^\top P^\top S^- P A_{\ell f}|  R_\ell^2}{\lambda_1(S)} + \frac{4 2 H_f}{\lambda_1(S) \gamma_o},  $$
$\dot{Z}(x_f) \leq 0$. This   implies that,  for all $t \in [0,T_\ell]$,  
\begin{align} \label{eqZb}
Z(x_f(t)) \leq \max \left\{ \sigma_{fo},  \sigma_f \right\}, \quad \sigma_f := \max \{ Z(x_f) : |x_f| \leq d_f\} \quad \sigma_{f o} := \max \{ Z(x_f) : |x_f| \leq r_o \}.  
\end{align}
In turn, for each $t \in [0,T_\ell]$, 
\begin{equation}\label{398}
|x_f(t)| \leq \bar{r}_o := \left[ \left( \sum^{n_f}_{i = 1} p_i \right) \min_i \{\underline{\alpha}_i\}  \right]^- \left( \max \left\{ \sigma_{fo},  \sigma_f \right\}\right).  
\end{equation} 
Clearly, the previous upper bound is uniform in $\gamma \geq \gamma_o$. 

Next, we focus on the solutions' behaviour  after $T_\ell$ (i.e., once $|x_\ell| \leq r_\ell$).  Given $\epsilon>0$, we see that, for all $\gamma \geq \gamma_o$ and for all $x_f$ and $x_\ell$ such that 
$$ |x_f| > \beta_1 := 1 + \frac{2 \bar p r_\ell}{\lambda_{1}(S)} + \sqrt{ \frac{\epsilon + H_f}{\gamma_o \lambda_{1}(S)}} \quad \text{and} \quad |x_\ell| \leq r_\ell, $$
after \eqref{296}, 
we conclude that
$
\dot Z(x_f) \leq - \epsilon.
$
Furthermore, $|x_\ell(t)| \leq r_\ell$ for all $t \geq T_\ell$, then the set 
$$  \mathcal{S}_{\sigma_1} := \{ x_f \in \mathbb{R}^{n_f} : Z(x_f) \leq \sigma_1 \}, \quad \sigma_1 :=  \max \{ Z(y) : y \in \mathcal{B}_{\beta_1} \}, \quad  \mathcal B_{\beta_1} :=  \{ x_f \in \mathbb{R}^{n_f} : |x_f| \leq \beta_1 \}, $$
is attractive and becomes forward invariant after time $T_\ell$.

Since $Z : \mathbb{R}^n \rightarrow \mathbb{R}_{\geq 0}$ is continuous and $\mathcal{B}_{\beta_1}$ is bounded, we conclude that $\sigma_1$ is well defined. 
As a result, the ultimate bound for $x_f(t)$ is 
$$ r_f = \left[ \left( \sum^{n_f}_{i = 1} p_i \right) \min_i \{\underline{\alpha}_i\} \right]^- (\sigma_1). $$  
Indeed, $Z(x_f) \leq \sigma_1$ implies $|x_f| \leq r_f$. 

Finally, as for $t\mapsto x_\ell(t)$ we give next an upperbound, denoted by $T_f (r_o,\gamma_o)$, on the time that the solutions to \eqref{foldyn}, with $\gamma \geq \gamma_o$ and starting from $\mathcal{B}_{r_o} := \{x_f \in \mathbb{R}^{n_\ell} : |x_f| \leq r_o \}$, take to reach $\mathcal{B}_{\beta_1} \subset \mathcal{S}_{\sigma_1}$. 

Let a solution $t \mapsto x_f(t)$ to \eqref{foldyn} starting from $x_f(0) \in \mathcal{B}_{r_o}$. Now,  we use the fact  $|x_f(T_\ell)| \leq \bar{r}_o$ with $\bar{r}_o$ coming from \eqref{398} and $\bar{r}_o$ is uniform in $\gamma$. As a result, along the solution $t \mapsto x_f(t)$, we have 
$\dot Z (x_{f}(t)) \leq  - \epsilon$ from $T_\ell$ and up to when it reaches $\mathcal{B}_{\beta_1}$ for the first time after $T_\ell$.
Hence, before reaching $\mathcal{B}_{\beta_1}$, we have 
\begin{equation}\label{459}
Z(x_{f}(t)) \leq  - \epsilon t + Z(x_{f}(T_\ell))
\end{equation} 
and, thus, using \eqref{eqZb}, we can take $T_f = T_\ell + \max\{\sigma_{fo}, \sigma_f \}/\epsilon$. 
 Clearly, $T_f$ and $r_f$ depend only on $(r_o,\gamma_o)$. Thus, the ultimate bounded guaranteed for the solutions of \color{black}\eqref{foldyn} is also uniform in $\gamma$.
\end{proof}

\begin{corollary}[Uniform boundedness]
Under Assumptions \ref{ass0} and \ref{ass1} the solutions of the closed-loop system in \eqref{eqcl} are globally uniformly bounded, {\it i.e.,} Property \ref{itemP1} holds.  
\end{corollary}

\begin{proof}
The statement of Theorem \ref{thm1} holds, therefore, \color{black} given $r_o>0$ and $\gamma_o>0$, for all $\gamma \geq \gamma_o$, we have
$$
|x_\ell(0)| \leq r_o \implies |x_\ell(t)| \leq r_\ell(\gamma_o) \qquad \forall t \geq T_\ell(r_o, \gamma_o).
$$
Furthermore, we were able to show that on the interval $[0, T_\ell(r_o, \gamma_o)]$, we have 
$$ W(x_{\ell}(t)) \leq  H_\ell T_\ell + W(x_{\ell}(0)).$$
Hence,  if we let $\sigma_\ell :=\max \{ W(y) : |y| \leq r_o \}$,  it follows that 
$$ |x_\ell(t)| \leq R_\ell := \left[ \min_i \{ \underline{\alpha}_i\} \right]^- \big( \sigma_\ell + H_\ell T_\ell + r_\ell \big)   \qquad \forall t \geq 0.   $$

Next, for the solutions to \eqref{foldyn}, for any $\gamma> \gamma_o$ and $|x_f(0)| \leq r_o$,  we know that
$$ |x_f(t)| \leq r_f \qquad \forall t \geq T_f(\gamma_o,r_o).  $$ 
At the same time, from the previous proof, we know that 
\begin{align*}
\dot Z(x_f) & 
\leq H_f - \gamma 
\left[ \lambda_{1}(S) |x_f|^2/2 - 2   |A_{\ell f}^\top P^\top S^- P A_{\ell f}|  R_\ell^2 \right].
\end{align*}
As a result, when
$$ |x_f|^2 > d^2_f :=  \frac{4 |A_{\ell f}^\top P^\top S^- P A_{\ell f}|  R_\ell^2}{\lambda_1(S)} + \frac{4 2 H_f}{\lambda_1(S) \gamma_o},  $$
then $\dot{Z}(x_f) \leq 0$. Hence, for each $t \geq 0$, 
\begin{align} 
Z(x_f(t)) \leq \max \left\{ \sigma_{fo},  \sigma_f \right\}, \quad \sigma_f := \max \{ Z(x_f) : |x_f| \leq d_f\}, \quad \sigma_{f o} := \max \{ Z(x_f) : |x_f| \leq r_o \}.  
\end{align}
In turn, for each $t \geq 0$, we have 
\begin{equation}
|x_f(t)| \leq R_f := \left[ \left( \sum^{n_f}_{i = 1} p_i \right) \min_i \{\underline{\alpha}_i\}  \right]^- \left( \max \left\{ \sigma_{fo},  \sigma_f \right\}\right).  
\end{equation} 
\end{proof}

\section*{Appendix}

The following lemma is proposed in \cite{9312975},  see also \cite{qu2009cooperative}.   
 
\begin{lemma} \label{lem1} 
Let $L \in \mathbb{R}^{n \times n}$ be the Laplacian matrix of a directed and strongly connected graph. 
Let $v_o := [v_1,v_2,...,v_n]^\top \in \mathbb{R}^{n}$ be the left eigenvector of $L$ associated to the null eigenvalue of $L$. 

Then, the vector $v$ has strictly positive entries and, for $V_o :=  \text{blkdiag}(v_o)$,  we have  
$\text{Ker} (V_o L + L^\top  V_o) = 
 \textrm{Span}~ (1_{n})$ and $V_o L + L^\top  V_o \geq 0$. 
\end{lemma}

The next result can be deduced from  \cite[Section 4.3.5]{qu2009cooperative}.  

\begin{lemma} \label{lem2Mf}
Let $M \in \mathbb{R}^{n \times n}$ be a  non-singular M-matrix.  Then, the matrices 
$$ S := R M + M^\top R \quad  \text{and} \quad  R :=  \text{blkdiag} \left\{ {M^\top}^- 
1_{n} \right\}   \left( \text{blkdiag} \left\{ M^- 
1_{n} \right\} \right)^- $$
are positive definite. 

\end{lemma}

\def\loria{Loria} \def\nesic{Ne\v{s}i\'{c}\,}\def\nonumero{\def\numerodeitem{}}


\end{document}